\documentclass[preprint,12pt]{elsarticle}
\usepackage[utf8]{inputenc}
\usepackage{graphicx}
\usepackage{geometry}
\usepackage{amsmath,amssymb,amsfonts,amsthm}
\title{Magic Polygons and Degenerated Magic Polygons: Characterization and Properties}
\newtheorem{prop}{Proposition}
\newtheorem{teorema}{Theorem}

\newtheorem{corolario}{Corollary}

\begin{document}

\begin{frontmatter}

\author[ed1]{Danniel Dias Augusto}
\ead{nieldiasguto@gmail.com}
\author[ed]{Josimar da Silva Rocha\corref{au2}}
\ead{jrocha@utfpr.edu.br}

\cortext[au2]{Corresponding author}

\address[ed1]{Coordena\c c\~ao da Matem\'atica, Universidade Estadual do Goi\'as, Unidade Universit\'aria de Formosa, 73807-250, Formosa - GO, Brazil} 

\address[ed]{Departamento de Matem\'atica, Universidade Tecnol\'ogica Federal do Paran\'a, \\ C\^ampus Corn\'elio Proc\'opio, \\
86300-000, Corn\'elio Proc\'opio - PR, Brazil}

\begin{abstract}
In this work we define Magic Polygons  $P(n,k)$ and Degenerated Magic Polygons $D(n,k) $ and we obtain their main properties, such as the  magic sum and the value  corresponding to the  root vertex. The  existence of magic polygons $P(n, k)$ and degenerated magic polygons $D(n, k)$ are discussed for certain values of  $n$ and $k.$ 

\end{abstract}


\end{frontmatter}

\section{Introduction}

Magic Squares have been known for a long time in different people and cultures that, sometimes, has attributed mystic meanings \cite{1,2,14}.  In addiction to being  used for recreational purposes, we can now find applications for magic squares in Physics, in Computer Science, in Image Processing and in Criptography \cite{4,5,8}, among others.   In this way, have been developed several  methods to construct magic squares that satisfies some  particular properties and some generalization have been created, as we can see in  \cite{3,7,9,10,11,13}. 

We can see in \cite{6} a generalization of the same idea of the representation magic squares of order $3$ using vertices, midpoints and the geometric center of a square, where we can find some properties and the condiction of existence of magic polygons and a construction for magic square of order 4, for each  $n$ even. In \cite{12} others similar structures was proposed.

In this work we will cover a generalization of the work proposed in  \cite{6} and we show that some valid properties of  magic polygons for a given  order sometimes are not valid in general. In addition, we will introduce another class of polygonal structure known as the class of degenerate magic polygons.

\section{ Magic Polygons   $P(n,k)$}

Let $\Omega$ be a set of  $\frac{k}{2}$ regular polygons on plane with $n$ sides and corresponding parallel sides and centered in a central point  $C.$ 

A magic polygon  $P(n, k)$ of  $n$ sides and order $k+1$ is a set of 
  $\frac{k^{2}n}{2} + 1$ points satisfying the following conditions:
  \begin{itemize}
\item[(i)] Points of magic polygon are labeled by distict values from  $1$ to  $\frac{k^{2}n}{2} + 1;$  
\item[(ii)] One point of a magic polygon is the central point $C$;
\item[(iii)] $\frac{kn}{2}$ points of magic polygon are vertices of the  $\frac{k}{2}$ regular polygons of  $\Omega;$ 
\item[(iv)] The magic polygon has  $k-1$ intermediate points on each edge of regular polygons in  $\Omega, $ which gives a total of  
$\displaystyle \frac{(k-1)nk}{2}$ intermediate points.

\item[(v)] Segments with  diametrically opposite ends of the larger polygon of  $\Omega$ intersecting the central vertex contain $k+1$ points of the magic polygon;

\item[(vi)] Segments with ends at two adjacent vertices of a polygon of  $\Omega$ contains  $k+1$ points of the magic polygon;
\item[(vii)] The sum of values corresponding  to the $k+1$ points on each segment defined in  (iv) and  (v) is a fixed value  $u,$ called of  {\bf magic sum.}  
\end{itemize}

In Figures \ref{P(4,4)n} e \ref{P(8,2)} we can see examples of Magic Polygons $P(4, 4)$ and $P(8,2), $ respectively.

 \begin{figure}[!htb]
\centering
\includegraphics[scale=0.3]{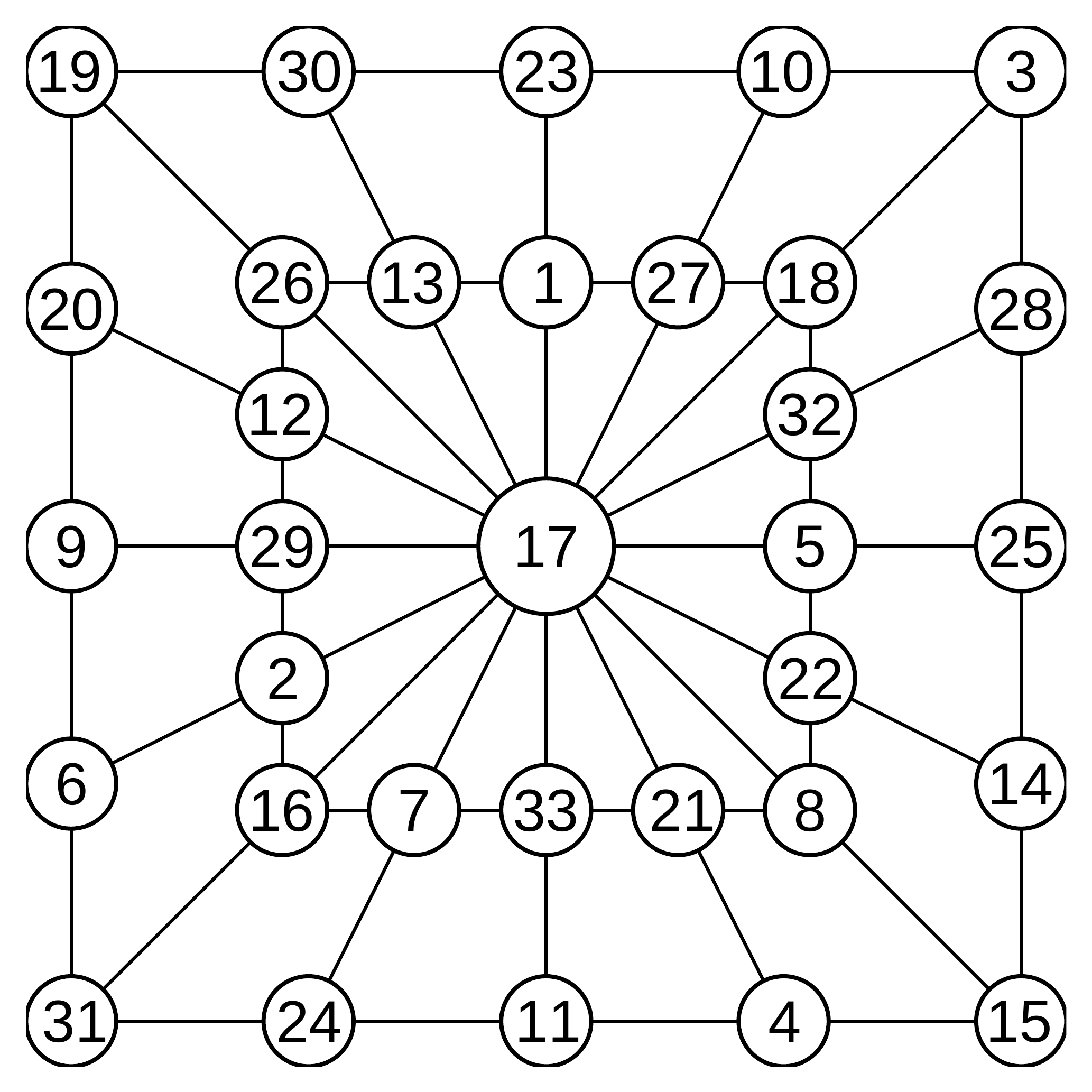}
\caption{Example of Magic Polygon  $P(4,4)$ }
\label{P(4,4)n}
\end{figure}

\begin{figure}[!htb]
\centering\includegraphics[scale=0.3]{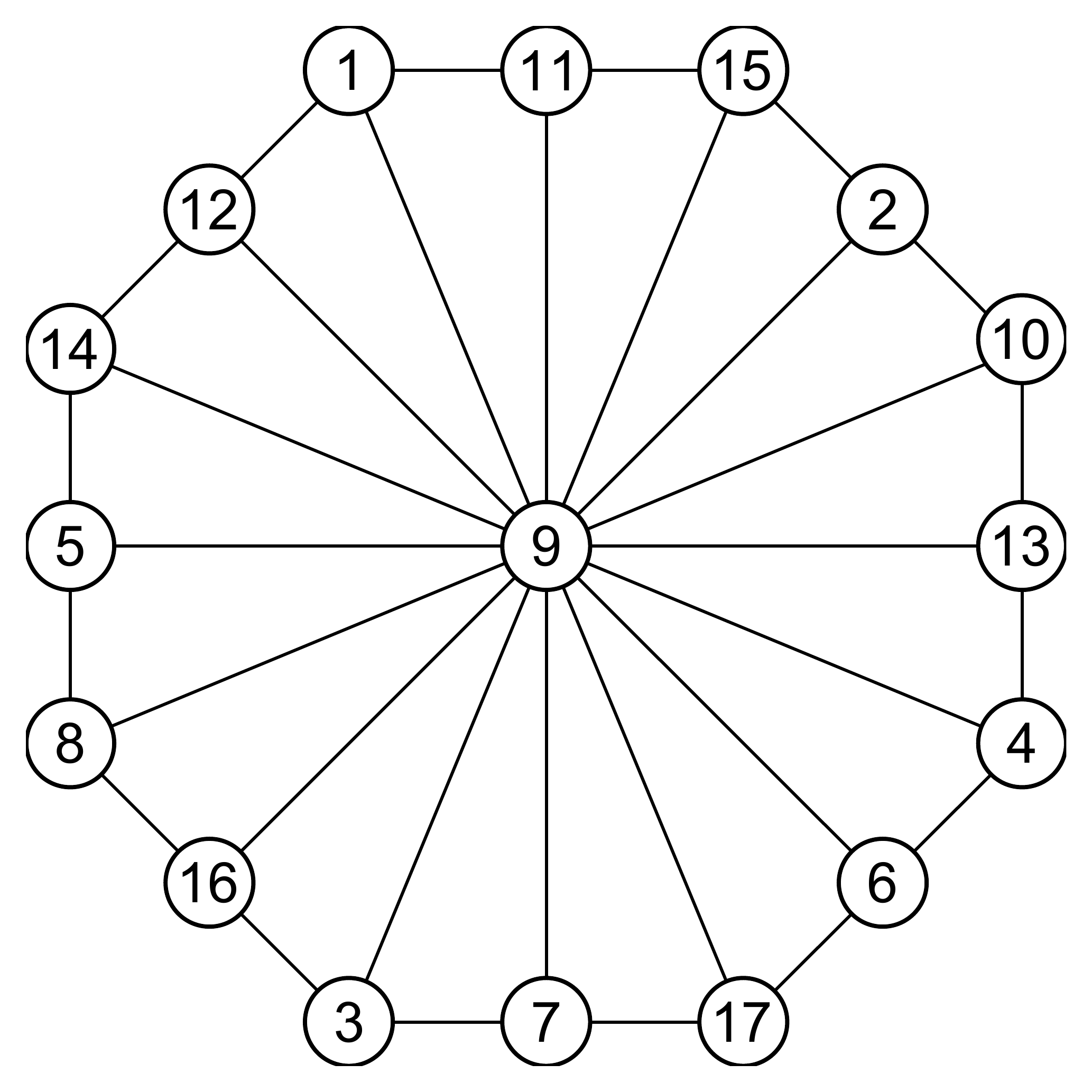}
\caption{Example of Magic Polygon  $P(8,2)$}
\label{P(8,2)}
\end{figure}

\begin{teorema} \label{teodeg}
  In a magic polygon  $P(n,k),$ we have the following properties:
 \begin{itemize}
 \item[(i)] the magic sum is  $(k+1)\frac{k^{2}n+4}{4};$
 \item[(ii)] the value corresponding to the root vertex is  $c = \frac{k^{2}n+4}{4};$
\item[(iii)] the sum $S_{j}$ of the values representing to   $j$-th points partitioning each edge on magic polygon  chosen clockwise is 
\[ S_{j} =  \frac{kn[k^{2}n+4]}{8} \]

\end{itemize}
\end{teorema}

\begin{proof}

Let  $x_{(t-1)nk+(i-1)k+j}$ be the value correspondig to the point  $P_{(t-1)nk+(i-1)k+j},$  $j$-th point of the  $i$-th edge of the  $i$-th regular polygon.

Each point  $P_{(t-1)nk+(i-1)k+j}$ of the magic polygon is labeled for a number 
\[ x_{(t-1)nk+(i-1)k+j}  \]
where  $t \in \{1, 2, \cdots, \frac{k}{2}\}, i \in \{ 1, 2, \cdots, n\} $ and $j\in \{1, \cdots, k\}$  and the root vertice is labeled by the number $c,$ then we have 
\begin{equation}
\frac{k^{2}n}{2} + 1
\end{equation} 
point in the magic polygon.

The sum of values correspondig points on the  $i$-th edge of  $t$-th regular polygon is given by 
\begin{equation}
x_{(t-1)nk+(i-1)k+1} + x_{(t-1)nk+(i-1)k+2} + \cdots + x_{(t-1)nk+ik} + x_{(t-1)nk+ik+1} = u,
\end{equation}
where  $u $ is the magic sum.

The sum of values correspondig points on the segments determined by the points  $P_{j+ik} $ and $P_{j+ (i-1)k + \frac{kn}{2} } $ is 

\begin{equation}\label{IIIeqg2}
\sum_{t=1}^{\frac{k}{2}} x_{(t-1)nk + (i-1)k+j} + \sum_{t=1}^{\frac{k}{2}} x_{(t-1)nk+(i-1)k+j+\frac{kn}{2}}+ c = u,
\end{equation}
 where $c$ is the value assigned to the root vertice and  $j + (i-1)k \leq \frac{kn}{2} - 1.$

Let 
\begin{equation}\label{IIIeqg3}
 S_{j} = \sum_{i=1}^{n}\sum_{t=1}^{\frac{k}{2}} x_{(t-1)nk+(i-1)k+j} 
 \end{equation}
for  $j \in \{1, 2, \cdots, k\}. $

By \eqref{IIIeqg2} and \eqref{IIIeqg3},
\begin{equation} \label{IIIeqg3-1}
S_{j} +\frac{nc}{2} = \frac{nu}{2},
\end{equation}
which implies  that 
\begin{equation}\label{IIIeqg3-2}
S_{j} = \frac{n(u-c)}{2}
\end{equation}

Adding equations involving points on perimeters of the polygons, we obtain
\begin{equation}\label{IIIeqg4}
2S_{1} + S_{2} + \cdots + S_{k}  = \frac{knu}{2}.
\end{equation}

Adding equations involving the root vertex of the magic polygon, we obtain
\begin{equation} \label{IIIeqg6}
S_{1} + S_{2} + \cdots + S_{k} + \frac{knc}{2} = \frac{knu}{2} 
\end{equation}
which implies that 

\begin{equation}\label{IIIeqg6-1}
S_{1}+S_{2} + \cdots + S_{k} = \frac{kn(u-c)}{2}
\end{equation}

Subtracting  \eqref{IIIeqg6-1} from  \eqref{IIIeqg4}, we have 

\begin{equation}\label{IIIeqg7}
S_{1} = \frac{knc}{2}
\end{equation}

By \eqref{IIIeqg3-2} and \eqref{IIIeqg7},  
\begin{equation} \label{IIIeqg8}
\frac{n(u-c)}{2} = \frac{knc}{2}.
\end{equation}

Therefore
\begin{equation} \label{IIIeqg9}
u = (k+1)c
\end{equation}

As  values corresponding to the points of the magic polygon are distinct values in  $\{1, 2, \cdots, \frac{k^{2}n}{2}+1\},$
it follows that 

\begin{equation}\label{IIIeqg10}
S_{1} + S_{2} + \cdots + S_{k} + c = \sum_{i=1}^{\frac{k^{2}n}{2}+1} i = \frac{\left(\frac{k^{2}n}{2}+1\right)\left(\frac{k^{2}n}{2} + 2\right)}{2}
\end{equation}

By \eqref{IIIeqg6-1} and \eqref{IIIeqg10}, it follows that 

\begin{equation}\label{IIIeqg11}
\frac{kn(u-c)}{2} + c = \frac{\left(\frac{k^{2}n}{2}+1\right)\left(\frac{k^{2}n}{2} + 2\right)}{2}
\end{equation}

By \eqref{IIIeqg9} and \eqref{IIIeqg11}, it follows that 
\begin{equation}\label{IIIeqg12}
\left(\frac{k^{2}n}{2}+1\right)c = \frac{\left(\frac{k^{2}n}{2}+1\right)\left(\frac{k^{2}n}{2} + 2\right)}{2}.
\end{equation}

Therefore
\begin{equation} \label{IIIeqg13}
c = \frac{k^{2}n+4}{4}
\end{equation}

Moreover,
\begin{equation}
S_{j} \stackrel{\eqref{IIIeqg3-2}}{=} \frac{n(u-c)}{2} \stackrel{\eqref{IIIeqg9}}{=}
\frac{nkc}{2} \stackrel{\eqref{IIIeqg13}}{=} \frac{nk\left[k^{2}n+4\right]}{8} 
\end{equation}
\end{proof}

\subsection{Constructing examples for $P(n,4)$}
The following result affords us a construction for  $P(n,4),$ provided that some conditions are satisfied:

\begin{teorema}\label{constn4}
Let two regular polygons with $n$ sides of distinct sizes centered on a central point $C$ whose sides are partitioned into $4$ segments by points such that $2n$ segments passing through the center point and cutting the larger polygon intercept the sides of the polygons at these points, satisfying the following conditions:

\begin{itemize}
\item[(i)] 
Each point $P_{4(t-1)n+4(i-1)+j}$ on partitions is labeled by   
$\displaystyle x_{4(t-1)n+4(i-1)+j}  \in \{1, 2, \cdots, 8n+1\};$
  
\item[(ii)] a central point is labeled by   $c = 4n+1;$

\item[(iii)] $x_{4(i-1)+j} + x_{4n + 4(i-1) + j} = 2(4n+1);$
\item[(iv)] $x_{4(t-1)n+4(i_{1}-1)+j_{1}} + x_{4(t-1)n+4(i_{2}-1)+j_{2}} \neq 2(4n+1),$ if  $(i_{1},j_{1}) \neq (i_{2}, j_{2}); $ 
\item[(v)] $x_{4(i-1)+1} + x_{4(i-1)+2} + x_{4(i-1)+3} + x_{4(i-1)+4} + x_{4(i-1)+5} = 5(4n+1);$

\noindent where $t \in \{1, 2\}, i_{1}, i_{2}, i \in \{ 1, 2, \cdots, n\} $ e $j_{1}, j_{2}, j\in \{1, 2, 3, 4\}.$
\end{itemize}
In this conditions, we obtain  $8n+1$ points that define a magic polygon  $P(n,4).$ 
\end{teorema}

\begin{proof} Let $i \in \{1, 2, \cdots, n-1\}. $ 

By  (iii), we have  $x_{4(i-1)+j} + x_{4n+4(i-1) + j} = 2(4n+1)$ for $j \in \{1, 2, 3, 4\}.$ 

Therefore
\begin{equation} \label{x1}
x_{4(i-1)+1} + x_{4n+4(i-1)+1} + \cdots + x_{4(i-1) + 4} + x_{4n + 4(i-1) + 4}  + x_{4(i-1) + 5} + x_{4n + 4(i-1) + 5} = 5\cdot 2 (4n+1) 
\end{equation}

By   \eqref{x1} and (v), we get  
\begin{equation} \label{x2}
5(4n+1) + x_{4n+4(i-1)+1} + x_{4n+4(i-1)+2} + x_{4n + 4(i-1) + 3} + x_{4n+4(i-1) + 4} + x_{4n + 4(i-1) + 5} = 10(4n+1)
\end{equation} 

Thus 
\begin{equation} \label{x3}
x_{4n + 4(i-1)+ 1} + x_{4n+4(i-1) + 2} + x_{4n + 4(i-1) + 3} + x_{4n + 4(i-1) + 4} + x_{4n + 4(i-1) + 5} = 5(4n + 1)
\end{equation} 

Consequently, the sum of values corresponding to the points on edges of the polygons is   $5(4n + 1).$ 

Therefore, by (iii) and  (ii), we get  

\begin{equation} \label{x4}
x_{4(i-1) + j} + x_{4n + 4(i-1) + j} + c + x_{4(i-1) + 2n + j} + x_{4n + 4(i-1) + 2n + j} = 2(4n+1) + 4n+1 + 2(4n+1) = 5(4n+1)   
\end{equation} 

Consequently, the sum of values corresponding to the points on the segments lying by central point is also $5(4n+1).$

Hence, by definition, we get  $P(n, 4).$  
\end{proof}

In Figure \ref{P(4,4)} we have an example of Magic Polygon $ P(4, 4) $ constructed by Theorem \ref{constn4}. An example of Magic Polygon $ P (4, 4) $ that can not be obtained by this construction can be seen in the figure \ref{P(4,4)n}.

\begin{figure}[!htb]
\centering
\includegraphics[scale=0.3]{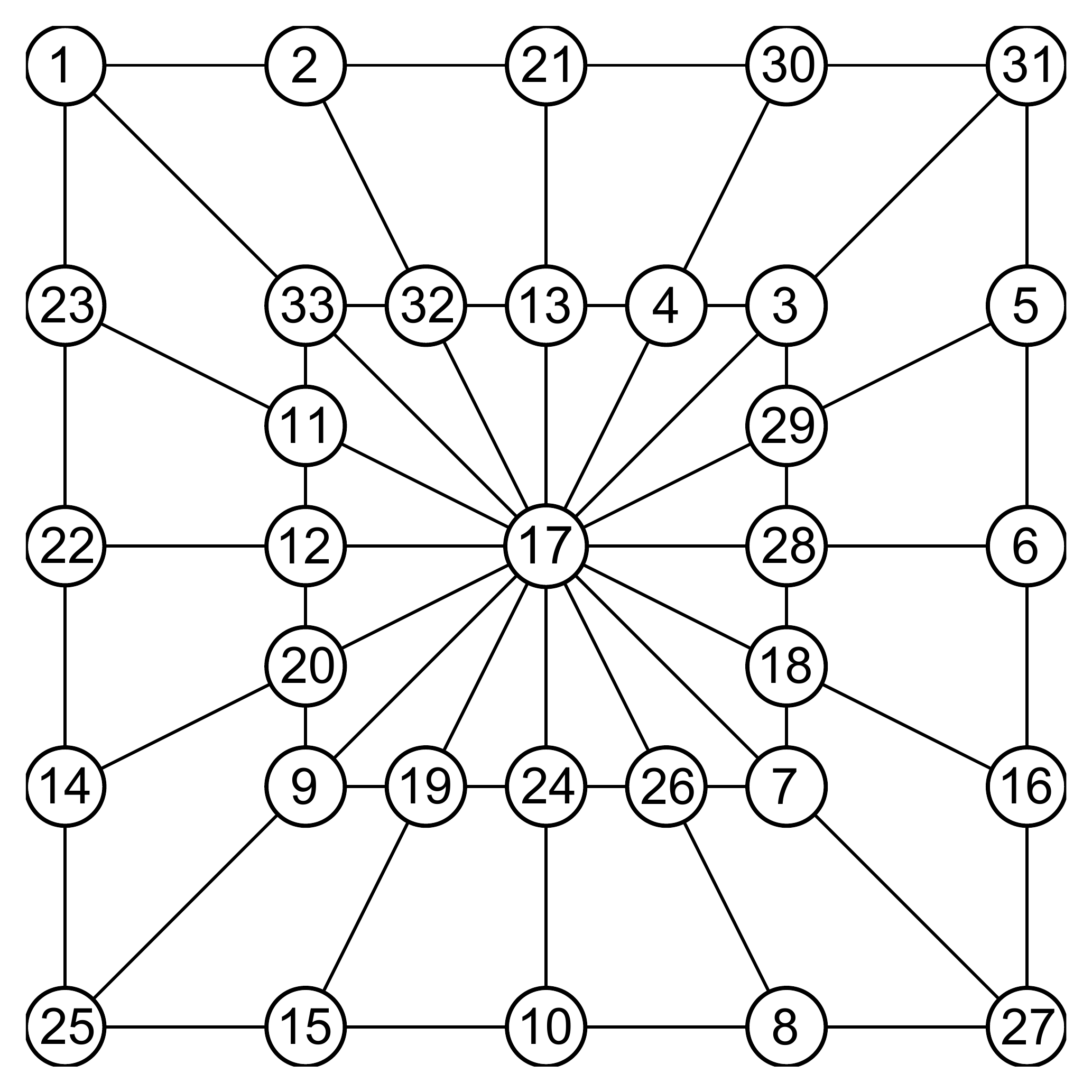}
\caption{Example of Magic Polygon   $P(4,4)$ using construction }
\label{P(4,4)}
\end{figure}

\subsection{The  particular case $P(n,2)$}

Although it is a particular case of the general case seen in the previous section, we will use another reasoning for a demonstration of the properties of the magic polygons $P(n,2).$

A {\bf magic polygon} $P(n,2)$ is formed by $2n+1$ points, consisting of the vertices, the midpoints of the edges, and the geometric center of a regular polygon of $n$ sides, which are labeled by numerical values from $1$ to $2n+1$, so that the sum of the values assigned to any three collinear pointes is constant, called {\bf magic sum.}

\begin{teorema}\label{teopm1}
 If $n$ is even, then a magic polygon  $P(n,2)$ has the following properties:
 \begin{itemize}
 \item[(i)] the magic sum is  \'e $3(n+1);$
 \item[(ii)] the value corresponding to the central point is   \'e $n+1;$
\item[(iii)] the sum  $S_{1}$ of the  values assigned  to  vertices of the magic polygon and the sum  $S_{2}$ of the values assigned to  midpoints of  edges of the magic polygon  satisfies $S_{1} = S_{2} = n(n+1).$
\end{itemize}
\end{teorema}

\begin{proof}

As a polygon with  $n$ sides has  $n$ vertex and  $n$ midpoints, including a central point, we obtain $2n+1$ points labeled using  distinct integers numbers  from  $1$ to $2n+1$.

Let  $x_{n}$ be the value assign to vertex $V_{n}$ and  let be   $y_{n}$ the values assign to the midpoint  $M_{n}$ whose endpoints are  $V_{n} $ and $V_{1}. $  For each  $i \in \{1, \cdots, n-1\}, $ let  $x_{i}$ be the value assigned to the vertex $V_{i}$ and let  $y_{i} $ be the value assigned to midpoint $M_{i}$ whose endpoints are  $V_{i}$ e $V_{i+1}.$

For each  $i \in \{1, \cdots, \frac{n}{2}\}, $ let $x_{i}^{\ast} $ be the value assigned to the point diametrically opposite to vertice  $V_{i}$ and  $y_{i}^{\ast} $ be the value assigned to the point  diametrically opposite to the midpoint  $M_{i}. $ 

Denoting by  $c$ the value assigned by the  central point  $C$ of the magic polygon and denoting by  $u$ the value of the magic sum, we obtain two sets of equations that define a magic polygon with $n$ sides, for an even  number $n$: the equations involving points on the same side  and equations involving points on simmetry axes of the magic polygon.

Analyzing equation involving points on the same side of a magic polygon, we obtain
    \begin{equation} \label{sis1}
    \left\{ \begin{array}{c}
        x_1+y_1+x_2=u \\
        x_2+y_2+x_3=u\\
    x_3+y_3+x_4=u \\
              \vdots \\
        x_n+y_n+x_1=u \\
    \end{array} \right. \end{equation}
    
Analyzing equations involving points on the same segment in the definition of a magic polygon, we obtain
\begin{equation}\label{sis2}
\left\{ \begin{array}{c}
    x_1+c+x_{1^*}=u \\
    y_1+c+y_{1^*}=u\\
    \vdots \\
x_{\frac{n}{2}}+c+x_{\frac{n}{2}^*} = u\\
    y_{\frac{n}{2}}+c+y_{\frac{n}{2}^*} = u\\
    \end{array} \right. \end{equation}
    
Let   $$S_1=\sum_{i=1}^n x_i  \mbox{ and }  S_2=\sum_{i=1}^n y_i$$   
     
Adding the equations of  \eqref{sis1}, we get
\begin{equation} \label{eq1}
2S_{1} + S_{2} = nu
\end{equation}

Adding the equations of \eqref{sis2}, we get

 \begin{equation} \label{eq2}
     S_1+S_2+nc=nu
 \end{equation}

Subtracting \eqref{eq2} from \eqref{eq1}, we get
\begin{equation} \label{eq3} 
    S_1=nc
\end{equation}

By  \eqref{eq3} and \eqref{eq2}, we get
\begin{equation} \label{eq4}
    S_2=nu-2nc
\end{equation}

As the values assigned to the   $n$ midpoints of the magic polygon with  $n$ sides are distinct values of the set $\{1, \cdots, 2n+1\} $ and the  sum of this values is  $S_{2}$, then  

\begin{equation}\label{eq5}
\sum_{i=1}^{n} i \le S_2 \le \sum_{i=n+2}^{2n+1} i
\end{equation}

In addition,
\begin{equation} \label{eq6}  
\sum_{i=1}^{n} i = \frac{n(n+1)}{2} 
\end{equation}
and 
\begin{equation}\label{eq7}
    \sum_{i=n+2}^{2n+1} i = \frac{3n(n+1)}{2}
\end{equation}

Therefore, by \eqref{eq5}, \eqref{eq6} and \eqref{eq7}, we get
\begin{equation}\label{eq8}
    \frac{n(n+1)}{2}\le S_2 \le \frac{3n(n+1)}{2}
\end{equation}

By \eqref{eq3} and \eqref{eq4}, we obtain

\begin{equation}\label{eq9}
    S_1+S_2+c=nu-c(n-1)
\end{equation}

Moreover, 
\begin{equation}\label{eq10}
S_1+S_2+c = \sum_{i=1}^{2n+1} i = \frac{(2n+1)(2n+2)}{2} = (2n+1)(n+1).
\end{equation}

Therefore, by \eqref{eq9} and \eqref{eq10}, we obtain the follow diophantine equation
\begin{equation}\label{eq11}
nu -c(n-1) = (2n+1)(n+1)
\end{equation}
whose general solution is 
\begin{equation}\label{eq12}
(u,c) = \left( (2n+1)(n+1) + (n-1)t, (2n+1)(n+1) + nt\right), \forall t \in \mathbb{Z}.
\end{equation}

By \eqref{eq12} and \eqref{eq4}, we obtain 
\begin{equation}\label{eq13}
\begin{array}{ll}
S_{2} & = nu - 2nc \\
& = n\left[ (2n+1)(n+1) + (n-1)t\right] - 2n\left[(2n+1)(n+1) + nt\right] \\
& = -n(2n+1)(n+1) - n(n+1)t \\ 
& = -(n+1)n(2n+1+t)
\end{array}
\end{equation}

By \eqref{eq13} and \eqref{eq8}, we get 
\begin{equation}\label{eq14}
\frac{n(n+1)}{2} \le -(n+1)n(2n+1+t) \le \frac{3n(n+1)}{2}.
\end{equation}

Simplyfing  \eqref{eq14}, we obtain

\begin{equation}\label{eq15}
\frac{1}{2} \le -(2n+1+t) \le \frac{3}{2} 
\end{equation}

It follows from   $-(2n+1+t) \in \mathbb{Z}$ and   \eqref{eq15} that  
$-(2n+1+t) = 1;$ hence 
\begin{equation} \label{eq16}
t = -2(n+1)
\end{equation}

By   \eqref{eq16} and \eqref{eq12}, 
\begin{equation}\label{eq17}
\begin{array}{ll}
(u,c) & = \left((2n+1)(n+1) + (n+1)t, (2n+1)(n+1) + nt\right) \\
& = \left( (2n+1)(n+1) -2(n+1)(n-1), (2n+1)(n+1) -2n(n+1)\right) \\
& = \left(3(n+1), (n+1)\right) \\
\end{array}
\end{equation}

Therefore,  the magic sum is  $u = 3(n+1)$ and the value assigned to the central point is  $c = n+1.$

Replacing   $u$ e $c$ from    \eqref{eq17} in  \eqref{eq3} and \eqref{eq4},
 we obtain
 \begin{equation}
    S_{1} = S_{2} = n(n+1).
 \end{equation}

\end{proof}

\begin{teorema} \label{teopm2}
If  $n$ is odd, then there is no magic polygon  $P(n,2)$. 
\end{teorema}

\begin{proof} Let  $n$  be odd and suppose that there is magic polygons with  $n$ sides.

If  $x_{1}, \cdots, x_{n}$ are integer numbers corresponding  vertices and  $y_{1}, \cdots, y_{n}, $ are integer numbers corresponding  midpoints of the magic polygon, where $y_{i}$ is the midpoint between the vertices  $x_{i}$ and $x_{i+1},$ we have 

\begin{equation} \label{2eq1}
x_{i} + y_{i} + x_{i+1} = u  
\end{equation}
and
\begin{equation}\label{2eq2}
x_{i} + y_{m+i-1} + c = u 
 \end{equation}
where $m = \frac{n+1}{2}, c$ is the value assigned to the root vertex of the magic polygon and  $x_{a} = x_{b}$ or  $y_{a} = y_{b}$  if  $a \equiv b \bmod n.$ 

By  \eqref{2eq2}, we obtain 
\begin{equation} \label{2eq3}
x_{i} = u - c - y_{m+i-1}
\end{equation}

Substituting  \eqref{eq3} into \eqref{2eq1}, we obtain

\begin{equation} \label{2eq4}
y_{i} = 2c - u + y_{m+i-1} + y_{m+i} 
\end{equation}
and
\begin{equation} \label{2eq5}
y_{i+1} = 2c - u + y_{m+i} + y_{m+i+1}
\end{equation}

By  \eqref{2eq4} and \eqref{2eq5}, we obtain

\begin{equation}\label{2eq6}
\begin{array}{ll}
y_{i+1} - y_{i} & = y_{m+i+1} - y_{m+i-1}  \\
& = y_{m+i+1} - y_{m+i} + y_{m+i} - y_{m+i-1} \\
& = y_{m+(m+i)+1} - y_{m+(m+i)-1} + y_{m+(m+i-1)+1} - y_{m+(m+i-1)-1} \\
& = y_{i+2} - y_{i} + y_{i+1} - y_{i-1}, \\
\end{array}  
\end{equation}

Taking    $n=3 $ and  $i = 1$  in the first equality of  \eqref{2eq6},  we obtain  
$y_{2} - y_{1} = y_{1} - y_{2}, $ which implies  $y_{1} = y_{2}.$ 

Taking  $n \geq 5$ and $i=2$ in  \eqref{2eq6}, we get     
$y_{1} = y_{4}.$

Therefore, we can not have a magic polygons with $n$ sides, for $n$ odd.

\end{proof}
 \begin{figure}[!htb]
\centering
\includegraphics[scale=0.3]{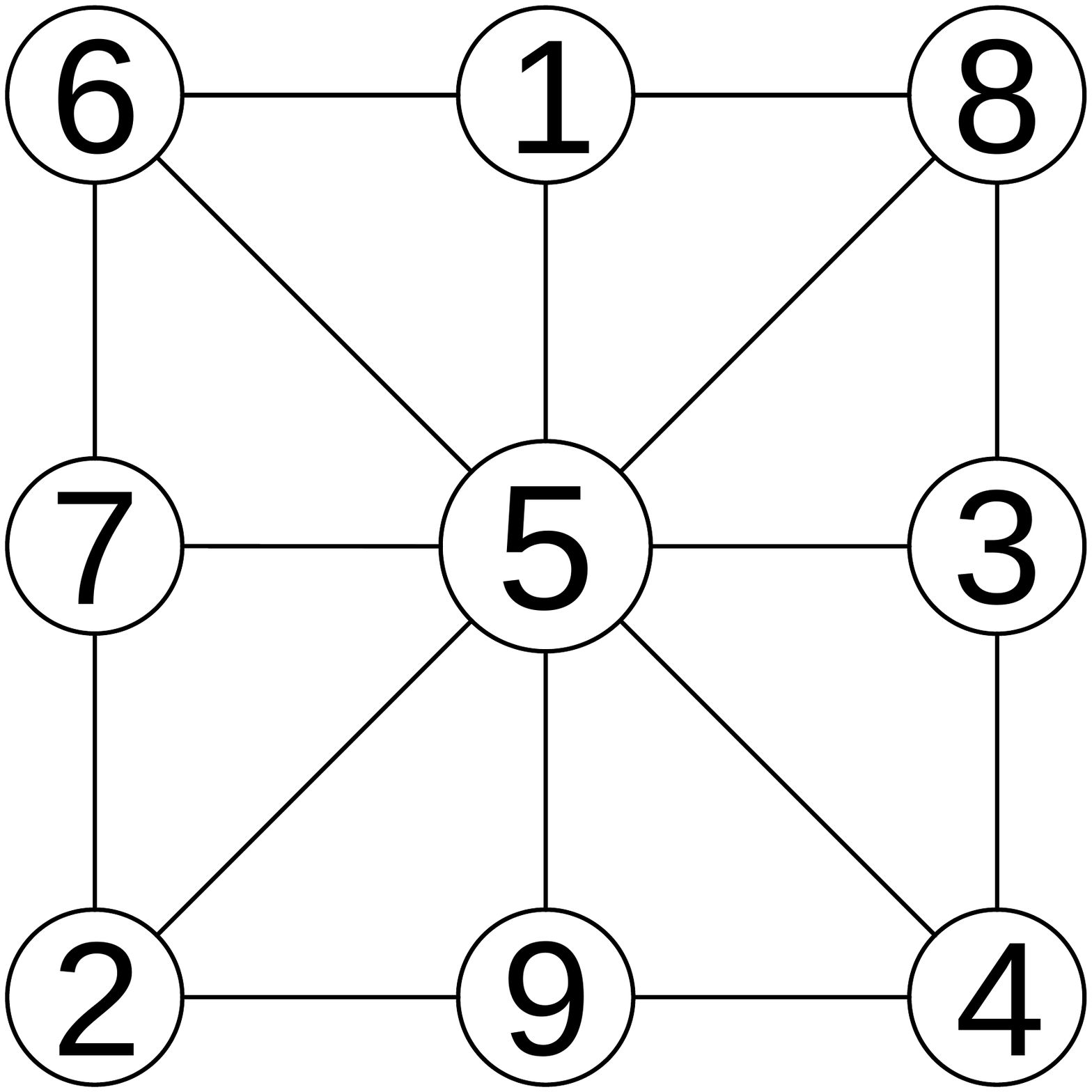}
\caption{Example of Magic Polygon  $P(4,2)$ }
\label{P(4,2)n}
\end{figure}

\begin{teorema} \label{constn2} If $n$ is an even number greater than or equal to 4, then there is a magic polygon $P(n,2)$.
\end{teorema}

\begin{proof} Figure \ref{P(4,2)n} shows the existence of Magic Polygons $P(4, 2).$  Let $n$ be even greater than or equal to 6 and consider a regular polygon with $n$ sides whose perimeter is indicated by the sequence of $n$ vertices clockwise $V_{1}, V_{2}, \cdots, V_{n}$. Thus, if for each $j \in \{1, \cdots, n\}, x_{j} $ is the value assigned to vertex  $V_{j}$ and  $y_{j}$ is the value assigned to midpoint of the side whose  ends are the  vertices  $V_{j} $ and $V_{j+1}$ of the magic polygon, then we obtain a magic polygon with  $n$ sides such that, for  $i \in \{ 1, 2, \cdots, \frac{n}{2} - 2\}, $ the values assigned to vertices of the magic polygon satisfy
 
 \[
\left\{ \begin{array}{ll}
x_{i} & =\left\{ \begin{array}{ll}
i + 1, & \text{if  $i$ is odd } \\
n + i + 2, & \text{if  $i$ is even } \\
\end{array} \right.  \\ \\
x_{i+\frac{n}{2}} & = \left\{ \begin{array}{ll}
2(n+1) - i - 1, & \text{ if $i$ is odd} \\
n - i, & \text{ if  $i$ is even} \\
\end{array} \right.  \\
x_{\frac{n}{2}-1} &  = n + 3 \\
x_{\frac{n}{2}} & = 1 \\
\end{array}\right. 
\]
and the values assigned to midpoints of magic polygon satisfy
\[ y_{j} = \left\{\begin{array}{ll} 
2(n+1) - x_{j} - x_{j+1}, & \text{ if  $j \in \{ 1, 2, 3, \cdots, n-1\} $} \\
2(n+1) - x_{1} - x_{n}, & \text{ if $j = n$} \\
\end{array}\right.  \]

The verification that this construction defines  a magic polygon can be seen in \cite{6}. 
\end{proof}

In Figure  \ref{P(4,2)} we have an example of Magic Polygon  $P(4,2)$ using construction in the proof of the Theorem \ref{constn2}.

\begin{figure}[!htb]
\centering
\includegraphics[scale=0.3]{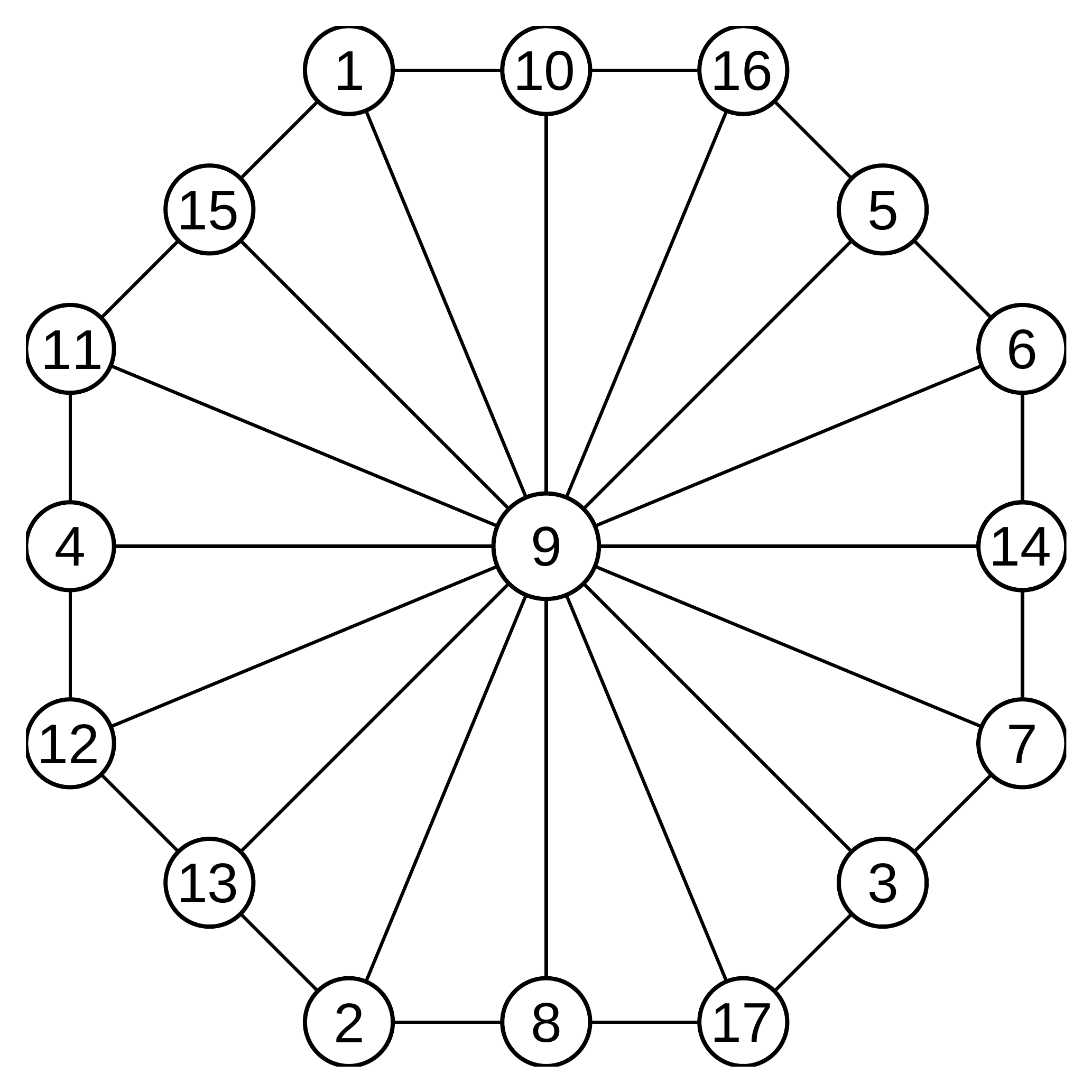}
\caption{Magic Polygon $P(4, 2)$ using construction }
\label{P(4,2)}
\end{figure}

\begin{prop} \label{proppm1}
There are no magic polygons whose values assigned to all vertices have odd parity.
\end{prop}
\begin{proof}
If the values assigned to all vertices are odd numbers, then, by definition of Magic Polygon, the values assigned to midpoints are even numbers. This contradicts the fact that the magic sum is an odd number, by Theorems  \ref{teopm1} and \ref{teopm2}.
\end{proof}

\begin{corolario}
There is no Magic Polygon  $P(n,2)$ whose values assigned to all midpoints are even numbers.
\end{corolario}
\begin{proof}
If all values assigned to midpoints of the Magic Polygon are even numbers, then, by definition of Magic Polygon, the values assigned to vertices of the Magic Polygon are odd numbers. This contradicts the  Proposition  \ref{proppm1}.
\end{proof}

\section{ Degenerated Magic Polygons $D(n,k)$}

Let  $\Omega$ be a set of  $k$ regular polygons with distinct sizes and with a common vertex C, called the root vertex.

A degenerated magic polygon  $D(n, k)$ with  $n$ sides and order  $k+1$ is a set of   $k^{2}(n-2) + k + 1$ points that satisfies the following conditions:

\begin{itemize}
\item[(i)]  $k^{2}(n-2) + k + 1$ points of degenerated magic polygon are assigned by distinct numbers from  $1$ to $k^{2}(n-2) + k + 1;$ 
\item[(ii)] One point is the root vertex $C;$
\item[(iii)] $kn-k+1$ points are the vertices of  $k$ regular polygons of $\Omega;$ 
\item[(iv)] each of the edges of the regular polygons not adjacent to the root vertices of  $\Omega$ have  $k-1$ intermediate points of the magic polygon beyond the vertices of the polygons, which gives a total of  $(n-2)(k-1)k $ intermediate points;
\item[(v)] Segments with one end in points at the border of the larger polygon of  $\Omega$ and other end on root vertex have  $k+1$ points of degenerated magic polygon;
\item[(vi)] Segments with ends on adjacent vertices of polygons of  $\Omega $ that do not contain the root vertex have  $k+1$ points of degenerated magic polygon;
\item[(vii)] The sum of values assigned to the $k+1$ points of the degenerated magic polygon in each of the  segments defined in (v) and (vi) is a fixed value   $u, $ called  the  magic sum. 
\end{itemize}

 Figures  \ref{D(5,2)} and \ref{D(5,3)} illustrate  the existence of Degenerated Magic Polygons.

 \begin{figure}[!htb]
\centering
\includegraphics[scale=0.3]{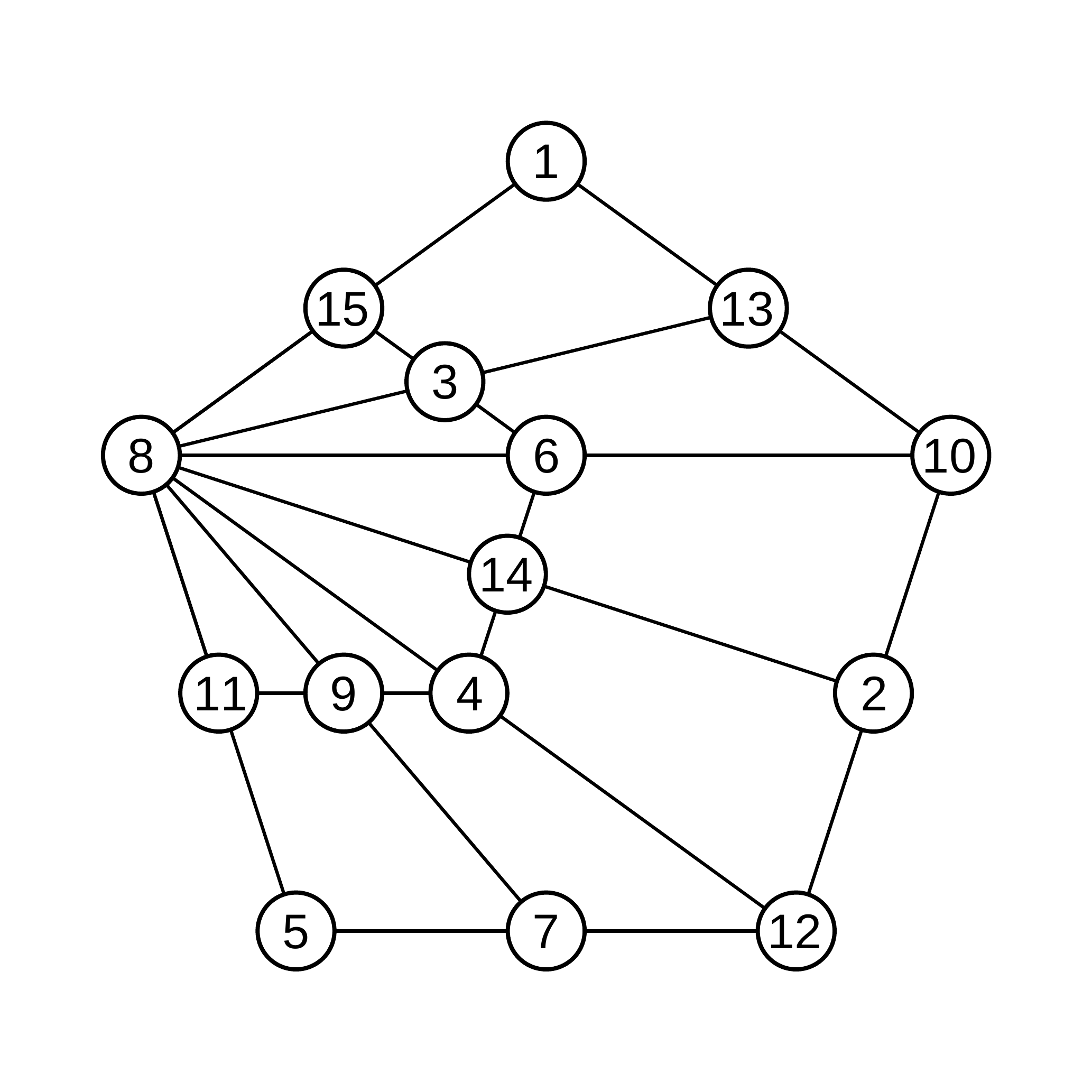}
\caption{Example of Degenerated Magic Polygon  $D(5,2)$ }
\label{D(5,2)}
\end{figure}

\begin{figure}[!htb]
\centering\includegraphics[scale=0.3]{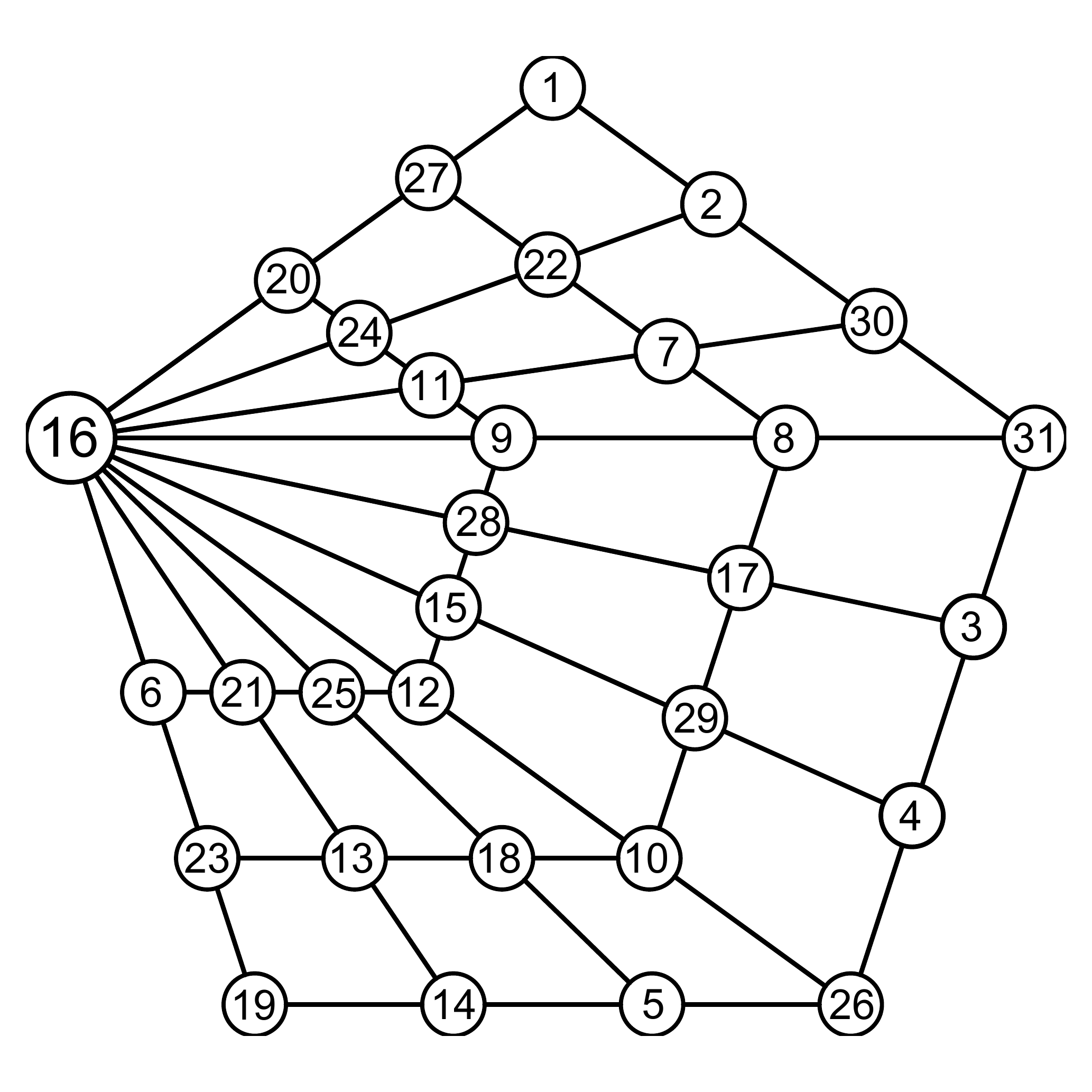}
\caption{Example of Degenerated Magic Polygon  $D(5,3)$}
\label{D(5,3)}
\end{figure}

 \begin{figure}[!htb]
\centering
\includegraphics[scale=0.3]{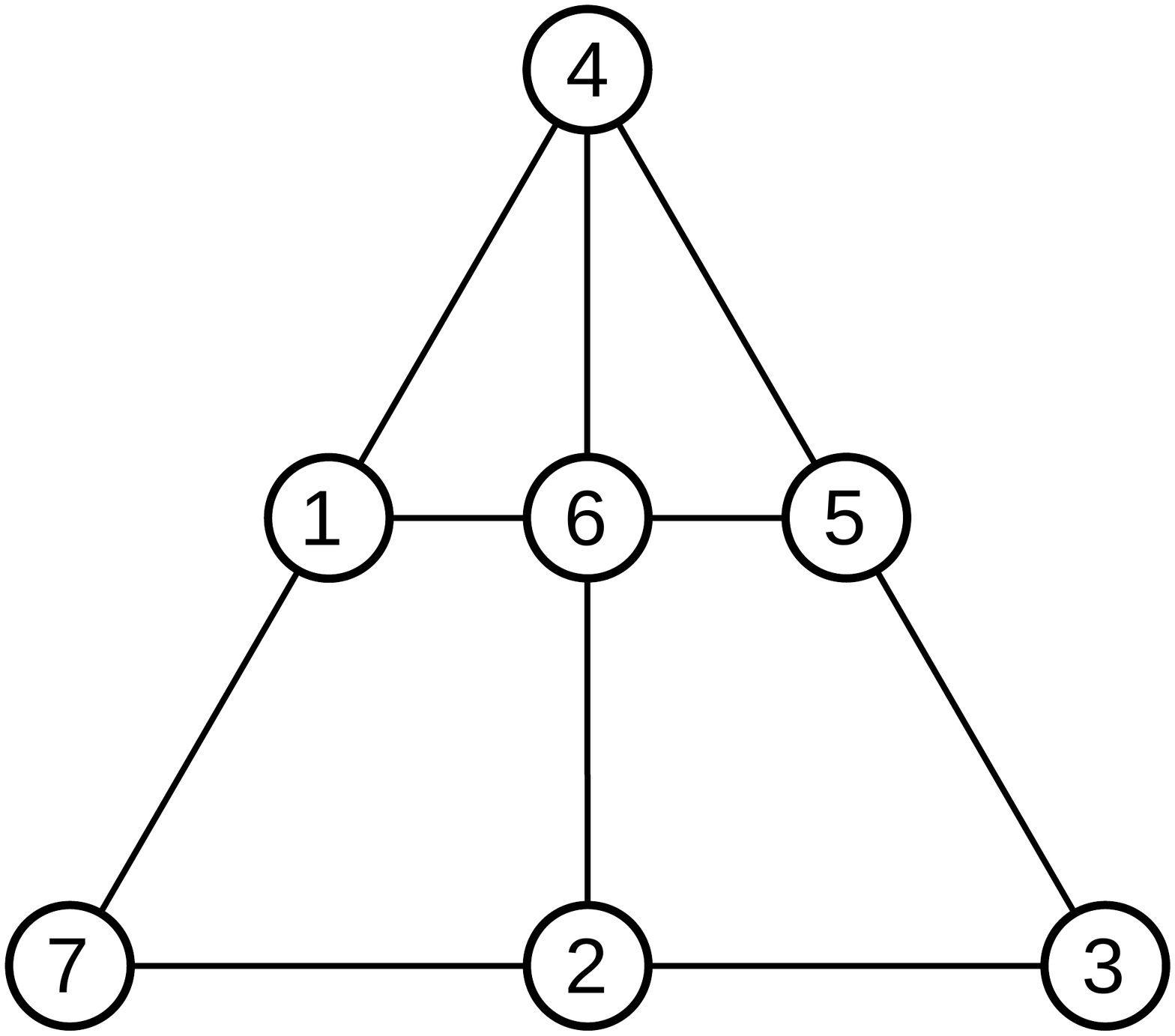}
\caption{Example of Degenerated Magic Polygon $D(3,2)$ }
\label{D(3,2)}
\end{figure}

\begin{teorema} \label{teodeg}
 A degenerated magic polygon  $D(n,k)$ has the following properties:
 \begin{itemize}
 \item[(i)] the magic sum is  $(k+1)\frac{k^{2}(n-2)+k+2}{2};$
 \item[(ii)] the value that corresponds to the root vertex is  $c = \frac{k^{2}(n-2)+k+2}{2};$
\item[(iii)] the sum $S_{j}$ of the values assigns to the  $j$-th points on the edges in the representation of the degenerated magic polygon satisfies
\[ S_{j} = (n-2)k\frac{k^{2}(n-2)+k+2}{2} \]

\end{itemize}
\end{teorema}

\begin{proof}
Let $P_{(t-1)k(n-2)+k(i-1) + j}$ be the   $j$-th  point of the  $i$-th edge of the $t$-th largest polygon that represents the degenerated magic polygon  $D(n,k), $ considering the clockwise direction, let 
\[ x_{(t-1)k(n-2)+k(i-1) + j}  \] be the value assigned to the point  $P_{(t-1)k(n-2)+k(i-1) + j},$   
\[ x_{tk(n-2) + 1} \] 
be the value assigned to the point    $P_{tk(n-2)+1},$
where  $j, t \in \{ 1, 2, \cdots, k\}$ and $  i \in \{ 1, 2, \cdots, (n-2)\}; $  the root vertex is labeled by the number   $c.$ So we have  
\begin{equation}
k^{2}(n-2) + k + 1
\end{equation} 
points in the degenerated magic polygon.

The equations involving  segments that don't  contain the root vertex are 
\begin{equation} \label{IIIeq1}
\left\{ \begin{array}{l}
x_{(t-1)k(n-2)+1} + x_{(t-1)k(n-2) + 2} + \cdots + x_{(t-1)k(n-2) + k + 1} = u\\
x_{(t-1)k(n-2) + k + 1} + x_{(t-1)k(n-2) + k + 2} + \cdots + x_{(t-1)k(n-2) + 2k + 1}= u\\
\vdots \\
x_{(t-1)k(n-2) + (i-1)k+1} + x_{(t-1)k(n-2) + (i-1)k + 2} + \cdots + x_{(t-1)k(n-2) + ik + 1} = u \\
\vdots \\ 
x_{(t-1)k(n-2)+(n-3)k + 1} + x_{(t-1)k(n-2) + (n-3)k + 2} + \cdots + x_{(t-1)k(n-2) + (n-2)k + 1} = u\\
\end{array}\right. \end{equation}

The equations involving segments containing  the root vertex are

\begin{equation}\label{IIIeq2}
\sum_{t=1}^{k} (x_{(t-1)k(n-2) + (i-1)k + j})  + c = u,
\end{equation}
where  
$ i \in \{ 1, 2, \cdots, n-2\} $ e $ j \in \{1, 2, \cdots, k\} $ and
\begin{equation}
\sum_{t=1}^{k} x_{tk(n-2) + 1} + c = u 
\end{equation}

Let 
\begin{equation}\label{IIIeq3}
 S_{1} = \sum_{i=1}^{n-1}\sum_{t=1}^{k} x_{(t-1)k(n-2)+k(i-1)+1} = (n-1)(u-c)
 \end{equation}
 and
\begin{equation} \label{IIIeq4} S_{j} = \sum_{i=1}^{n-2}\sum_{t=1}^{k} x_{(t-1)k(n-2) + k(i-1)+j} = (n-2)(u-c)
\end{equation}
for $j \in \{2, \cdots, k\}. $ 

By adding equations involving  perimeters in the degenerated magic polygon, we obtain 
\begin{equation}\label{IIIeq5}
2S_{1} + S_{2} + \cdots + S_{k} + 2c = [k(n-2)+2]u,
\end{equation}
or, equivalently,
\begin{equation}\label{IIIeq5-1}
2S_{1} + S_{2} + \cdots + S_{k} = [k(n-2)+2]u - 2c
\end{equation}

By adding equations involving the root vertex of the degenerated magic polygon, we obtain
\begin{equation} \label{IIIeq6}
S_{1} + S_{2} + \cdots + S_{k} + (k(n-2)+1)c = (k(n-2)+1)u 
\end{equation}
or, equivalently,
\begin{equation}\label{IIIeq6-1}
S_{1} + S_{2} + \cdots + S_{k} = (k(n-2)+1)(u-c)
\end{equation}

By subtracting \eqref{IIIeq6-1} from  \eqref{IIIeq5-1}, we obtain 

\begin{equation}\label{IIIeq7}
S_{1} = u + (k(n-2)-1)c
\end{equation}

By \eqref{IIIeq3} and \eqref{IIIeq7},  
\begin{equation} \label{IIIeq8}
(n-1)(u-c) = u + (k(n-2)-1)c.
\end{equation}

Hence
\begin{equation} \label{IIIeq9}
u = (k+1)c
\end{equation}

As the values assigned to the points of the degenerated magic polygon are distinct values of the set  $\{1, 2, \cdots, k^{2}(n-2)+k+1\},$ we obtain

\begin{equation}\label{IIIeq10}
S_{1} + S_{2} + \cdots + S_{k} + c = \sum_{i=1}^{k^{2}(n-2)+k+1} i = \frac{(k^{2}(n-2)+k+1)(k^{2}(n-2)+k+2)}{2}
\end{equation}

By \eqref{IIIeq6-1} and \eqref{IIIeq10},
\begin{equation}\label{IIIeq11}
[k(n-2)+1](u-c) + c = \frac{(k^{2}(n-2)+k+1)(k^{2}(n-2)+k+2)}{2}
\end{equation}

By \eqref{IIIeq9} and \eqref{IIIeq11}, 
\begin{equation}\label{IIIeq12}
[k^{2}(n-2)+k+1]c = \frac{(k^{2}(n-2)+k+1)(k^{2}(n-2)+k+2)}{2},
\end{equation}
that implies 
\begin{equation} \label{IIIeq13}
c = \frac{k^{2}(n-2)+k+2}{2}
\end{equation}

\end{proof}

\begin{corolario} If $k$ and $n$ are positive integer numbers, such that 
$k$ is odd and  $n$ is even, then there is no degenerated magic polygon  $D(n,k).$
\end{corolario}

\begin{proof} By Theorem \ref{teodeg},  
\[ c = \frac{k^{2}(n-2) + k + 2}{2}. \]
 Therefore,  if  $k$ is odd and  $n$ is even, then $c$ is not a integer number. Hence, there is no degenerated magic polygon of order $k+1$ with $n$ vertex for $k$ odd and $n$ even.
  \newline
\end{proof}

\begin{teorema}\label{constpmd2} If $n$ is an integer number greater than or equal to 3,  then there is a degenerated magic polygon $D(n,2)$.
\end{teorema}

\begin{proof} Let  $C$ be the root vertex of a degenerated magic polygon  $D(n, 2),$  $(C, V_{1}, \cdots, V_{n-1}) $ be the sequence of vertices of the largest polygon and  $(C, V_{1}^{\ast}, \cdots, V_{n-1}^{\ast})$ be the sequence of vertices of smallest polygon, both on clockwise direction.

For each  $j \in \{1, 2, \cdots, n-2\}, $ we consider   $M_{j} $  the point of degenerated magic polygon between  $V_{j} $ and  $V_{j+1}$ and  $M_{j}^{\ast}$  the point of the degenerated magic polygon between  $V_{j}^{\ast}$ and $V_{j+1}^{\ast}. $ 

Thus, if   $c$ is the value assigned to the root vertex $C$ and for each $j\in \{ 1, 2, \cdots, n-2\}, z_{j}$ is the value assigned to the vertex  $V_{j}, z_{j}^{\ast}$ is the value assigned to the vertex $V_{j}^{\ast}, m_{j}$ is the value assigned to the point $M_{j}$ and $m_{j}^{\ast}$ is the value assigned to the point  $M_{j}^{\ast}, $  then, the following conditions are satisfied:

\begin{equation}\label{pd1}
z_{j} + m_{j} + z_{j+1} = u
\end{equation} 

\begin{equation}\label{pd2}
z_{j}^{\ast} + m_{j}^{\ast} + z_{j+1}^{\ast} = u
\end{equation}

and
\begin{equation}\label{pd2}
z_{j} + z_{j}^{\ast} + c = u
\end{equation} 

Setting, for each $j \in \{1, 2, \cdots, n-1\}, $ 
 
 \[
\left\{ \begin{array}{ll}
z_{j} & =\left\{ \begin{array}{ll}
j , & \text{if  $j$ is odd } \\
2n + j -3, & \text{if  $j$ is even } \\
\end{array} \right.  \\ \\
z_{j}^{\ast}  & = \left\{ \begin{array}{ll}
4(n-1) - j, & \text{ if $j$ is odd} \\
2n - j-1, & \text{ if  $j$ is even} \\
\end{array} \right.  \\
m_{j} & = 4(n-1)-2j \\
m_{j}^{\ast} & = 2j \\
\end{array}\right. 
\]
then, we obtain a degenerated magic polygon $P(n,2).$

In fact, the condictions on segments are satisfied, because
\[ z_{j}+m_{j} + z_{j+1} = \left\{ \begin{array}{cl}
j + 2n+j+1-3 + 4(n-1) - 2j = 6(n-1), & \text{ if $j $ is odd}  \\
2n + j - 3 + 4(n-1) - 2j + j + 1 = 6(n-1), & \text{ if $j$ is even}\\
  \end{array}\right. \]
  
 \[ z_{j}^{\ast} + m_{j}^{\ast} + z_{j+1}^{\ast} = \left\{ \begin{array}{cl}
4(n-1) - j +2j + 2n - j -1 - 1 =  6(n-1), & \text{ if $j $ is odd}  \\
2n - j - 1 + 2j + 4(n-1) - j - 1 =  6(n-1), & \text{ if $j$ is even}\\
  \end{array}\right. \]  
  
  \[ z_{j} + z_{j}^{\ast} + c = \left\{ \begin{array}{cl}
  j + (4(n-1) - j) + 2(n-1) = 6(n-1), & \text{ if $j$ is odd} \\
 (2n + j - 3) + (2n - j - 1) + 2(n-1) = 6(n-1) & \text{ if $j$ is even} \\
\end{array} \right. 
  \]
  
  \[ m_{j} + m_{j}^{\ast} + c = 
  4(n-1) - 2j + 2j + 2(n-1) = 6(n-1)\]
  
  In addition, all values assigned to the points are different because
  
  \[ \begin{array}{cl} A_{1} &  = \{ z_{j} \mid \text{ $j$ is odd and } j \in \{1, 2, \cdots, n - 1\} \} \\
  & = \left\{ \begin{array}{cl}
  \{ 1, 3, 5, \cdots, n - 2\}, & \text{ if $n$ is odd} \\
  \{ 1, 3, 5, \cdots, n - 1\}, & \text{ if $n$ is even} \\
  \end{array} \right. \\
  \end{array} \]
  \[\begin{array}{cl} B_{1} & = \{ z_{j} \mid \text{ $j$ is even and } j \in \{1, 2, \cdots, n - 1\} \} \\
  & = \left\{ \begin{array}{cl}
  \{ 2n-1, 2n+1, 2n+3, \cdots, 3n-4\}, & \text{ if $n$ is odd} \\
  \{ 2n-1, 2n+1,  2n+3, \cdots, 3n - 5\}, & \text{ if $n$ is even} \\
  \end{array} \right. \\
  \end{array} \]
  \[\begin{array}{cl}  A_{2} & = \{ z_{j}^{\ast}  \mid \text{ $j$ is odd and } j \in \{1, 2, \cdots, n - 1\} \} 
  \\
  & = \left\{ \begin{array}{cl}
  \{ 4n-5, 4n - 7, 4n - 9, \cdots, 3n-2\}, & \text{ if $n$ is odd} \\
  \{ 4n-5, 4n -7, 4n-9, \cdots, 3n - 3\}, & \text{ if $n$ is even} \\
  \end{array} \right. \\
  \end{array} \]
  \[ \begin{array}{cl}  B_{2} & = \{ z_{j}^{\ast}  \mid \text{ $j$ is even and } j \in \{1, 2, \cdots, n - 1\} \} \\
  & = \left\{ \begin{array}{cl}
  \{ 2n-3, 2n-5, 2n-7, \cdots, n\}, & \text{ if $n$ is odd} \\
  \{ 2n-3, 2n-5, 2n-7, \cdots,  n+1\}, & \text{ if $n$ is even} \\
  \end{array} \right. \\
  \end{array}  \]
  \[ \begin{array}{cl} 
  C_{1} & = \{ m_{j}  \mid  j \in \{1, 2, \cdots, n - 2\} \} \\
  & = 
  \{ 4n-6, 4n-8, 4n-10, \cdots, 2n \} \\
  \end{array} \]
    \[\begin{array}{cl}  C_{2} & = \{ m_{j}^{\ast}   \mid  j \in \{1, 2, \cdots, n - 2\} \} \\
    & = 
  \{ 2, 4, 6, 8, 10, 12, \cdots, 2n-4\} \\
  \end{array} \]
   \[ D=\{ c \} = \{ 2(n-1)\} \]
   satisfy
   \[ |A_{1}| = |A_{2}| = \left\{ \begin{array}{cl}
 \frac{n-1}{2}, & \text{ if $n$ is odd} \\
 \frac{n}{2}, & \text{ if $n$ is even} \\
 \end{array} \right. \]
 \[     |B_{1}| = |B_{2}| = \left\{ \begin{array}{cl}
 \frac{n-1}{2}, & \text{ if $n$ is odd} \\
 \frac{n-2}{2}, & \text{ if $n$ is even} \\
 \end{array} \right. \]
   \[ |C_{1}| = |C_{2}| = n - 2, \]
   \[ |A_{1}| + |A_{2}| + |B_{1}| + |B_{2}| + |C_{1}| + |C_{2}| + |D| = 4n -5 \]
   
   \[ \begin{array}{ll}  
   & A_{1} \cup A_{2} \cup B_{1} \cup B_{2} \cup C_{1} \cup C_{2} \cup D \\
    = & \{j \in \mathbb{Z} \mid 1 \leq j \leq 4n - 5\} \\
    = & \{ 1, 2, \cdots,  4n - 5\}\\
    \end{array} 
    \]   
    and 
   \[  |A_{1}\cup A_{2} \cup B_{1} \cup B_{2}\cup C_{1} \cup C_{2} \cup D| = 4n-5 \] 
\end{proof}

In Figure \ref{D(5,2)} we have an example of Magic Polygon $ D(5, 2) $ constructed in the proof of the Theorem \ref{constpmd2}. An example of Magic Polygon $ D(5, 2) $ that can not be obtained by this construction can be seen in the figure \ref{D(5,2)c}.

 \begin{figure}[!htb]
\centering
\includegraphics[scale=0.3]{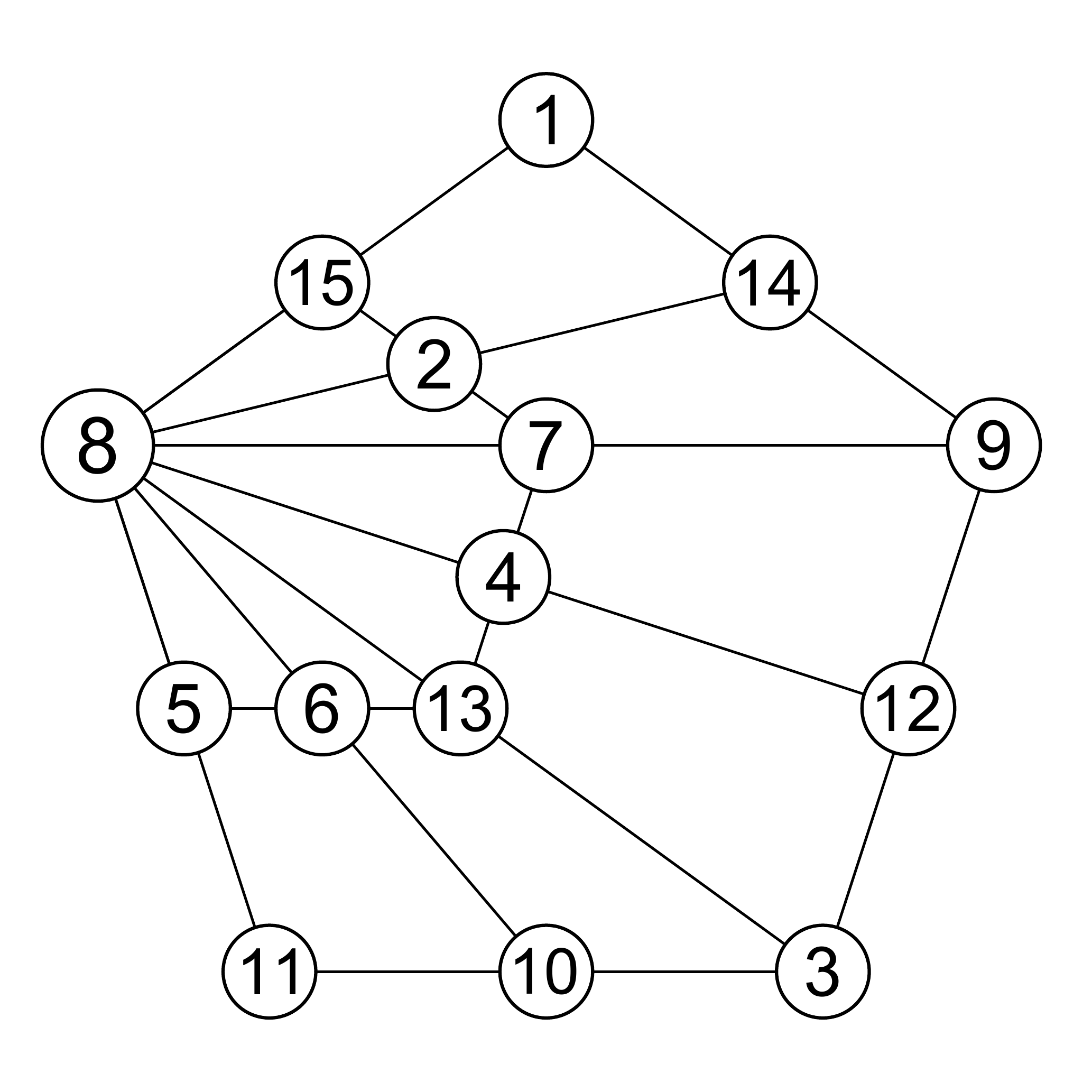}
\caption{Example of Degenerated Magic Polygon  $D(5,2)$ using construction }
\label{D(5,2)c}
\end{figure}


\begin{thebibliography}{sssd}


\bibitem[1]{1} W.R. Andress, Basic properties of pandiagonal magic squares, Amer. Math. Monthly 67 (1960) 143–152.
\bibitem[2]{2} S. Cammann, The evolution of magic squares in China, J. Am. Oriental Soc. 80 (1960) 116–124.
\bibitem[3]{3} C.Y.J. Chan et al., A construction of regular magic squares of odd
order, Linear Algebra and its Applications 457 (2014) 293–302

\bibitem[4]{4} K.L. Chu, S.W. Drury, G.P.H. Styan, G. Trenkler, Magic Moore–Penrose inverses and philatelic magic square with special emphasis
on the Daniels–Zlobec magic square, Croatian Oper. Res. Rev. 2 (2011) 4–13.

\bibitem[5]{5} G. Ganapathy, K. Mani, Add-on security model for public-key cryptosystem based on magic square implementation, in: Proc.
World Congress on Engineering and Computer Science 1 WCECS, 2009.
\bibitem[6]{6} V. Jakicic,  R.  Bouchat, Magic Polygons and their properties, http://www.arxiv.org/abs/1801.02262v1, 2018.

\bibitem[7]{7} Y. Kim, J. Yoo, An algorithm for constructing magic squares, Discrete Applied Mathematics 156 (2008) 2804–2809.

\bibitem[8]{8} P.D. Loly, Franklin squares: a chapter in the scientific studies of magical squares, Complex Systems 17 (2007) 143–161.
.
\bibitem[9]{9} R.B. Mattingly, Even order regular magic squares are singular, Amer. Math. Monthly 107 (2000) 777–782.
\bibitem[10]{10} R.P. Nordgren, New constructions for special magic squares, Int. J. Pure Appl. Math., in press.
\bibitem[11]{11} K. Ollerenshaw, D.S. Brée, Most-Perfect Pandiagonal Magic Squares: Their Construction and Enumeration, The Institute of Mathematics and its Applications, Southend-on-Sea, UK, 1998.

\bibitem[12]{12} C.A. Pickover, The Zen of Magic Squares, Circles, and Stars, second printing and first paperback printing, Princeton University
Press, Princeton, NJ, 2003 (original printing and e-book: 2002).
\bibitem[13]{13} C. Planck, Pandiagonal magic squares of orders 6 and 10 without minimal numbers, Monist 29 (1919) 307–316.
\bibitem[14]{14} B. Rosser, R.J. Walker, The algebraic theory of diabolic magic squares, Duke Math. J. 5 (1939) 705–728.
 

\end{thebibliography}
\end{document}